\documentclass[12pt,leqno]{article}

\newif\ifsattoc\sattoctrue
\newread\testfl\immediate\openin\testfl=\jobname.toc
    \ifeof\testfl\sattocfalse\fi\closein\testfl

\usepackage[lite]{amsrefs}
\usepackage{latexsym,enumerate}

  \usepackage[notref, notcite]{}
\usepackage{amssymb, amsfonts, eucal,  amsmath,amsthm}
\usepackage{graphicx,float}
\usepackage[list=true]{subcaption}

\DeclareCaptionSubType*[arabic]{figure}
\usepackage{color}
\definecolor{added}{rgb}{0, 0, 1}
\definecolor{deleted}{rgb}{1, 0, 0}

\newtheorem{theorem}{Theorem}[section]

\newtheorem{lemma}[theorem]{Lemma}

\newtheorem{remark}[theorem]{Remark}

\newtheorem{corollary}[theorem]{Corollary}

\newtheorem{openproblem}[theorem]{Open Problem}

\newcommand{\sect}[1]{\section{#1} \setcounter{equation}{0} }

\newcounter{ca}
\newcommand{\norm}[1]{\left\|#1\right\|}
\newcommand{\EE}{\widetilde E}
\newcommand{\dn}{\delta_n}

\newcommand{\Poly}{\mathbb P}
\newcommand{\Pn}{\Poly_n}
 \newcommand{\ec}{\end{comment}}
\newcommand{\bc}{ \begin{comment}
 }

\newcommand{\opp}{\scalebox{0.75}{$\overset{\oplus}{+}$}}
\newcommand{\opm}{\scalebox{0.75}{$\overset\oplus\ominus$}}

\newcommand{\NN}{{\mathcal N}}

\newcommand{\andd}{\quad\mbox{\rm and}\quad}

\newcommand\w{{\omega}}
\def\be  {\begin{equation}}
\def\ee  {\end{equation}}
\def\ba  {\begin{eqnarray}}
\def\ea  {\end{eqnarray}}
\def\baa {\begin{eqnarray*}}
\def\eaa {\end{eqnarray*}}
\newenvironment{comment}[2]
{\bgroup\vspace{7pt}
\begin{tabular}{|p{5in}|}
\hline \qquad \bf \footnotesize Comment -- to be deleted in the final version \\
\hline
\quad\sl\footnotesize #1#2} {\\ \hline \end{tabular}
\vspace{7pt}\indent\egroup}

\def\updots{\mathinner{\mkern
1mu\raise 1pt \hbox{.}\mkern 2mu \mkern 2mu \raise
4pt\hbox{.}\mkern 1mu \raise 7pt\vbox {\kern 7 pt\hbox{.}}} }

\def \esssup{\mathop{\rm ess\: sup}\nolimits}

\newcommand{\C}{C}

\newcommand{\R}{\mathbb R}

\newcommand{\N}{\mathbb N}

\newcommand{\Y}{\mathbb Y}

\renewcommand{\a}{\alpha}

\newcommand{\ineq}[1]{{\rm(\ref{#1})}}

\newcommand{\ie}{{\em i.e.}}
\newcommand{\eg}{{\em e.g. }}

\newcommand{\bpic}{
\begin{center}
}

\newcommand{\epic}{
\endpspicture
\end{center}
}

\newcommand{\st}{\;  :  \;}

\newcommand{\thm}[1]{Theorem~\ref{#1}}
\newcommand{\lem}[1]{Lemma~\ref{#1}}
\newcommand{\cor}[1]{Corollary~\ref{#1}}

\newcommand{\sectio}[1]{Section~\ref{#1}}

\title{{\sc Uniform and pointwise shape preserving approximation (SPA) by algebraic polynomials: an update}\thanks{{\it AMS classification:} 41A10, 41A17, 41A25. {\it Keywords
and phrases:} Approximation by polynomials,  moduli of smoothness, constrained approximation}}

\author{K. A.  Kopotun\thanks{Department of Mathematics, University of
Manitoba, Winnipeg, Manitoba, R3T 2N2, Canada ({\tt
kirill.kopotun@umanitoba.ca}). Supported by NSERC of Canada.} ,
D. Leviatan\thanks{Raymond and Beverly Sackler School of Mathematical
Sciences, Tel Aviv University, Tel Aviv 6139001, Israel ({\tt
 leviatan@tauex.tau.ac.il}).}\ \ and I. A. Shevchuk\thanks
{Faculty of Mechanics and Mathematics, Taras
Shevchenko National University of Kyiv, 01033 Kyiv, Ukraine ({\tt
shevchuk@univ.kiev.ua}).} }

\begin{document}

\maketitle

\abstract{
It is not surprising that one should expect that the degree of constrained (shape preserving) approximation
be worse than the degree of unconstrained approximation. However, it turns out that, in certain cases, these degrees are the same.

The main purpose of this paper is to provide an update to our 2011 survey paper.  In particular, we discuss recent uniform estimates in comonotone approximation, mention recent developments and state several open problems in the (co)convex case,   and reiterate that co-$q$-monotone approximation with $q\ge 3$ is completely different from comonotone and coconvex cases.

Additionally, we show that, for each function $f$ from $\Delta^{(1)}$, the set of all monotone functions on $[-1,1]$,  and every $\alpha>0$, we have

\[
\limsup_{n\to\infty}
\inf_{P_n\in\Pn\cap\Delta^{(1)}}\norm{ \frac{n^\a(f-P_n)}{\varphi^\alpha} }
\le c(\a) \limsup_{n\to\infty}  \inf_{P_n\in\Pn}\norm{ \frac{n^\a(f-P_n)}{\varphi^\alpha} }
\]
 where $\Pn$ denotes the set of algebraic polynomials of degree $<n$,
$\varphi(x):=\sqrt{1-x^2}$, and $c=c(\alpha)$ depends only on $\alpha$.

 }

\sect{Introduction}\label{secintro}

The main purpose of this paper is to provide an update to our paper \cite{KLPS}.

Let $C[-1,1]$ be the space of continuous functions on $[-1,1]$
equipped with the uniform norm $\|\cdot\|$ (since there is no confusion, we will use the same notation for the $\esssup$-norm of $L_\infty$ functions), and denote by
$\Delta^{(q)}$  the set of all $q$-monotone functions $f\in C[-1,1]$. In
particular, $\Delta^{(1)}$ and $\Delta^{(2)}$
are, respectively, the sets of nondecreasing and convex
functions which are continuous on $[-1,1]$. If $\Pn$ is the space of
algebraic polynomials of degree $<n$, then
\begin{eqnarray*}
E_n(f)=\inf_{P_n\in\Pn}  \|f-P_n\|  \andd
E_n^{(q)}(f)=\inf_{P_n\in\Pn\cap\Delta^{(q)}} \|f-P_n\|
\end{eqnarray*}
denote, respectively, the errors of best unconstrained and $q$-monotone approximation of a function $f$ by polynomials from $\Pn$.

In the sequel,   $c=c(\dots)$ denotes positive constants which   depend only on the parameters inside the parentheses and are otherwise absolute. Note that all constants $c$'s are assumed to be different even if they appear in on the same line.

Given a function $f\in\Delta^{(q)}$, it is clear that $E_n(f)\le E_n^{(q)}(f)$ and, in 1969,  Lorentz and Zeller \cite{LZ} proved that the inverse inequality $E_n^{(q)}(f)\le c
E_n(f)$  is not true in general even with the constant $c=c(f,q)$ that depends on   $f$. Specifically,
they proved that there exists an $f\in\Delta^{(q)}$ for which
\[
\limsup_{n\to\infty}\frac{E_n^{(q)}(f)}{E_n(f)}=\infty.
\]
Moreover, if $f\in C^{(q)}[-1,1]$, then it is easy to prove that there exists a constant $c=c(q)$ such
that
$$
E_n(f)\le\frac c{n^q}E_{n-q}(f^{(q)}).
$$
It is even easier to prove that, for $f\in\Delta^{(q)}\cap C^{(q)}[-1,1]$,
\be\label{qmon}
E_n^{(q)}(f)\le\frac2{q!}E_{n-q}(f^{(q)}).
\ee
However, comparing the last two estimates we see that we have lost an order of $n^q$.

It turns out (see \cite{LS}) that there exists a constant $c=c(q)>0$ such that, for each $n>q$, there is a non-polynomial $f=f_n\in\Delta^{(q)}\cap C^{(q)}[-1,1]$, for which
\[
E_n^{(q)}(f)\ge cE_{n-q}(f^{(q)}).
\]
Thus, \ineq{qmon} may not, in general, be improved.

Despite all this, it is known (see \eg  \cite{KLPS}*{p. 53}) that,
for every   $\a>0$ and $f\in\Delta^{(q)}$, $q=1,2$, we have
\be\label{compare}
\sup \{ n^\a E_n^{(q)}(f) \st  n\ge1\} \le c(\a) \sup\{ n^\a E_n(f) \st  n\ge1\}.
\ee
Note that estimates of this type  are, in general, invalid if $q\ge 3$, and we discuss this in more detail in \sectio{lastsection}. Hence, we only concentrate on the cases $q=1$ and $q=2$ in this section.

A natural question is whether similar results are valid for
piecewise monotone and piecewise convex functions, \ie,  functions
which are allowed to change their monotonicity or convexity
$s<\infty$ times in the interval $(-1,1)$.

In order to give precise statements we need some additional
definitions. 
Let $\Y_s$, $s\in\mathbb N$, be the set of all
collections $Y_s:=\big\{y_i\big\}_{i=1}^s$ of points $y_i$, such
that $-1<y_s<\cdots<y_1<1$. We augment the set by $y_{s+1}:=-1$ and
$y_0:=1$. For $Y_s\in\Y_s$ denote by $\Delta^{(q)}(Y_s)$
the set of all  piecewise $q$-monotone functions $f\in\C[-1,1]$ that change $q$-monotonicity at the points in $Y_s$. More precisely, $f\in \Delta^{(q)}(Y_s)$ iff
$f$ is $q$-monotone in the intervals
$[y_{2i+1},y_{2i}]$, $0\le i\le \lfloor s/2 \rfloor$, and $-f$ is $q$-monotone in the intervals $[y_{2i},y_{2i-1}]$, $1\le i\le \lfloor (s+1)/2 \rfloor$.
In particular, for $q=1,2$, these are all $f\in\C[-1,1]$  that change monotonicity/convexity
 at the points in $Y_s$, and are nondecreasing/convex on $[y_1,1]$.
For convenience, we also include the case of $q$-monotone functions (that is, $s=0$) in this notation by putting $Y_0:=\emptyset$. Hence,
$\Delta^{(q)}  = \Delta^{(q)}(Y_0)$,  $E_n^{(q)}(f) = E_n^{(q)}(f,Y_0)$, etc.

Denote by
\[
E_n^{(q)}(f,Y_s)=\inf_{P_n\in\Pn\cap\Delta^{(q)}(Y_s)} \|f-P_n\|
\]
the error  of best co-$q$-monotone approximation of a
function $f\in \Delta^{(q)}(Y_s)$.

Rewriting \ineq{qmon} for the new notions, the question now is whether the inequality
\be\label{12mon}
\sup \{ n^\a E_n^{(q)}(f,Y_s) \st  n\ge1\} \le c(\a,s)\sup\{ n^\a E_n(f) \st  n\ge1\}
\ee
is valid or not.

Surprisingly, the answer is rather different for $q=1$ and $q=2$ (see \cite{KLPS}*{Section 15} for references).
If $q=2$, then \ineq{12mon} is valid if $s=1$ and $\a \in (0,4)\cup (4, \infty)$, and is invalid in all other cases (\ie, if $s\ge 2$ or if $s=1$ and $\a=4$).
If $q=1$, then \ineq{12mon} is valid if  $\a>0$ and $\a \not\in A_s$, where
\be\label{exceptional}
A_s:=\left\{\alpha \st \alpha=j, \; 1\le j\le s-1,\quad \text{or}\quad \alpha=2j, \; 1\le j\le s \right\},
\ee
and is invalid if $\a \in A_s$.

 We emphasize that \ineq{12mon} becomes valid for all $s\ge 1$ and $q=1,2$ if its left-hand side is replaced by
$ \sup \{ n^\a E_n^{(q)}(f,Y_s) \st  n\ge \NN^* \}$ with $\NN^*$ depending on $\a$ and $Y_s$ (see \cite{KLPS}*{Section 15} for detailed discussions).

The next natural question is whether similar results are valid for pointwise estimates, but we only concentrate on the case $s=0$ since
 analogous results with $s\ge 1$ have still not been completely resolved, while some conclusions can be made from the known pointwise results (see \eg \cite{KLPS}*{Tables 21, 22, 27 and 28}).

Let
\be\label{phi}
\varphi(x):=\sqrt{1-x^2}, \quad \dn(x) := \varphi(x) + n^{-1}  \andd \rho_n(x) := n^{-1} \dn(x) ,
\ee
and denote
\[
E_{n,\alpha}(f):=\inf_{P_n\in\Pn}\norm{ \frac{f-P_n}{\varphi^\alpha} }, \quad
\EE_{n,\alpha}(f):=\inf_{P_n\in\Pn}\norm{ \frac{f-P_n}{\dn^\alpha} }
\]
\[
E^{(q)}_{n,\alpha}(f):=\inf_{P_n\in\Pn\cap\Delta^{(q)}}\norm{ \frac{f-P_n}{\varphi^\alpha} } \andd
\EE^{(q)}_{n,\alpha}(f):=\inf_{P_n\in\Pn\cap\Delta^{(q)}}\norm{ \frac{f-P_n}{\dn^\alpha} } .
\]

Clearly, for $\a>0$,
\[
\EE_{n,\alpha}(f) \le E_{n,\alpha}(f) \le E^{(q)}_{n,\alpha}(f) \andd
\EE_{n,\alpha}(f) \le  \EE^{(q)}_{n,\alpha}(f)    \le E^{(q)}_{n,\alpha}(f) .
 \]

Note   that we usually refer to estimates involving
$E_{n,\alpha}(f)$ and $E^{(q)}_{n,\alpha}(f)$ as ``interpolatory results'' since it is necessary for approximating polynomials $P_n$ to interpolate $f$ at $\pm 1$ in order for these quantities to be finite.

Now, for every   $\a>0$ and $f\in\Delta^{(q)}$, $q=1,2$, it follows from  \eg \cite{KLPS}*{Tables 6 and 7} and \lem{modulus} below (with $\NN=1$) that
\be\label{pointnotints0}
\sup \{ n^\a \EE_{n,\a}^{(q)}(f) \st  n\ge1\} \le c(\a)\sup\{ n^\a \EE_{n,\a}(f) \st  n\ge1\} .
\ee

Interpolatory estimates are different. First, it follows from
  \cite{dy} ($q=1$) and \cite{L} ($q=2$) that the following inequality  is valid if $f\in\Delta^{(q)}$, $q=1,2$,  and $\a\in (0,2)$.
\be\label{pointints0}
\sup \{ n^\a E_{n,\a}^{(q)}(f) \st  n\ge1\} \le c(\a)\sup\{ n^\a E_{n,\a}(f) \st  n\ge1\} ,
\ee
and the following lemma implies that \ineq{pointints0} is not valid if $\a>2$. Note that it is still an open question if, for $q=1,2$,  \ineq{pointints0} is valid  if $\a=2$.

\begin{lemma}[see \eg \cite{kls2}*{(1.5)} if $q=1$ and \cite{kls-interconvex} if $q=2$] \label{thmneg1}
Let $q=1$ or $q=2$. For any $r\in\N$ and each $n\in\N$, there is a   function $f\in\C^{(r)}[-1,1]\cap \Delta^{(q)}$, such that for every   polynomial $P_n\in\Pn\cap\Delta^{(q)}$ and any positive on $(-1,1)$ function $\psi$ such that $\lim_{x\to \pm 1} \psi(x)=0$, either
\be \label{glswineqcon}
\limsup_{x\to -1} \frac{|f(x)-P_n(x)|}{\varphi^2(x) \psi(x)} = \infty \quad \mbox{\rm or}\quad
\limsup_{x\to 1} \frac{|f(x)-P_n(x)|}{\varphi^2(x)\psi(x)} = \infty .
\ee
\end{lemma}

In fact, it follows from \lem{thmneg1} that, in the case $q=1,2$, even the estimate
\be\label{pointnotints0new}
\sup \{ n^\a E_{n,\a}^{(q)}(f) \st  n\ge \NN^*\} \le c(\a)\sup\{ n^\a E_{n,\a}(f) \st  n\ge1\}
\ee
is not valid in general if $\a>2$ and $\NN^*$ is any natural number which is independent of $f$.

It turns out that, if $q=1$ and $\NN^*$ is allowed to depend on $f$ then \ineq{pointnotints0new} is valid for any $f\in\Delta^{(1)}$ and all $\a>0$, and it is still an open problem if
the same conclusion can be made in the case $q=2$ (see also Open Problem~\ref{openconvex} below).

In fact, the following stronger result  holds.

\begin{theorem} \label{pointnew}
Given a function $f\in\Delta^{(1)}$, $\alpha>0$ and $\NN\in\N$, there exists a constant $\NN^*=\NN^*(\a,f,\NN)\in\N$ such that
\be\label{pointnotints0newmain}
\sup \{ n^\a E_{n,\a}^{(1)}(f) \st  n\ge \NN^*\} \le c(\a)\sup\{ n^\a \EE_{n,\a}(f) \st  n\ge\NN\} .
\ee
\end{theorem}

\thm{pointnew} with $0<\alpha<2$, $\NN=1$ and $\NN^*=1$  was proved by DeVore and Yu (see \cite{dy}*{Theorem 2}).
If $\a \ge 2$, then \thm{pointnew} is a corollary of \thm{pointwise} and \lem{modulus} (see \sectio{sect2}).

 \thm{pointnew} immediately implies the following.

\begin{corollary}\label{point} For each function $f\in\Delta^{(1)}$ and every $\alpha>0$ we have
\be \label{corin}
\limsup_{n\to\infty}n^\alpha E_{n,\alpha}^{(1)}(f)\le c(\a) \limsup_{n\to\infty}n^\alpha \EE_{n,\alpha}(f) .
\ee
\end{corollary}

We emphasize that the right-hand sides of \ineq{pointnotints0newmain} and \ineq{corin} are given in terms of ``non-interpolatory'' quantities $\EE_{n,\alpha}(f)$ which are smaller than ``interpolatory'' $E_{n,\alpha}(f)$. In other words, we obtain interpolatory estimates for monotone approximation with the same order as non-interpolatory unconstrained estimates.

This paper is organized as follows. In \sectio{sect2}, we discuss interpolatory pointwise estimates for monotone approximation and the inverse result that will yield \thm{pointnew}.
\sectio{comsec}   summarizes  recent results  in the comonotone case. Several open problems for (co)convex approximation are stated and briefly discussed in \sectio{convecsection}.
Finally, we show in \sectio{lastsection} that estimates of the above type for co-$q$-monotone with $s\ge1$ and $q\ge3$ are, in general, invalid.

\sect{Pointwise estimates and proof of \thm{pointnew}}\label{sect2}

The following direct interpolatory pointwise result for monotone approximation by algebraic polynomials was recently proved in \cite{kls2}  (in fact, Theorem~1.2 in \cite{kls2} is a slightly stronger than we we state here).

\begin{theorem}\label{pointwise} For every $r\ge1$, there is a constant $c(r)$ such that, for each function $f\in\Delta^{(1)} \cap C^{(r)}[-1,1]$, there are a number $\NN=\NN(f,r)$
and a sequence $\{P_n\}_{n\ge \NN}$ of polynomials $P_n\in\Pn\cap\Delta^{(1)}$, satisfying
\be\label{101}
|f(x)-P_n(x)|\le c(r)\left(\frac{\varphi(x)}n\right)^r\omega_2\left(f^{(r)},\frac{\varphi(x)}n\right),\quad x\in[-1,1] .
\ee
\end{theorem}

Now, \thm{pointnew} follows immediately from \thm{pointwise} and the following lemma which
 easily follows from the classical Dzyadyk inverse theorem for approximation by algebraic polynomials (see \eg  \cite{ds}*{p. 381}   for the references).

\begin{lemma}\label{modulus}
Suppose that $r\in \N_0$,  $\alpha\in(0,2)$ and $\NN\in\N$. If, for $f:[-1,1]\mapsto\R$
and every $n\ge \NN$, there is a polynomial $P_n\in\Pn$  such that
\be\label{inverse}
|f(x)-P_n(x)|\le\rho^{r+\alpha}_n(x),\quad x\in[-1,1],
\ee
then $f\in C^{(r)}[-1,1]$ and
\[
\omega_2(f^{(r)},t)\le c(r,\alpha)t^{\alpha}+c(r,  N)t^2 E_{r+2}(f).
\]
\end{lemma}


\begin{proof}  Without loss of generality assume that $\NN> r+2$.
Denote by $R_{r+2}\in\Poly_{r+2}$ the polynomial of best uniform approximation of the function $f$. Put $R_\NN:=P_\NN$, if
\[
\norm{ (f-P_\NN)\rho_\NN^{-r-\alpha} }  \le \norm{ (f-R_{r+2})\rho_\NN^{-r-\alpha} } ,
\]
otherwise put $R_\NN:=R_{r+2}$ .
Also, let
\[
g:=f-R_\NN,\quad Q_n:=P_n-R_\NN\quad\text{for}\quad n\ge \NN, \quad \text{and}\quad Q_n\equiv0 \quad\text{for}\quad1\le n<\NN.
\]
Then, for all $n\ge1$,
\[
|g(x)-Q_n(x)|\le\rho^{r+\alpha}_n(x),\quad x\in[-1,1],
\]
and so the classical Dzyadyk inverse theorem  implies that $g\in C^{(r)}[-1,1]$ and
\[
\omega_2(g^{(r)},t)\le c(r,\alpha)t^{\alpha}.
\]
By the Dzyadyk inequality for the derivatives of algebraic polynomials (see \eg \cite{ds}*{(7.1.11)} and note that the constant $M$ there actually depends on $\lfloor s \rfloor$), we have
\[
\norm{(R_\NN^{(r+2)}-R_{r+2}^{(r+2)})\rho_\NN^{2-\alpha}}
\le c_1 \norm{ (R_\NN-R_{r+2})\rho_\NN^{-r-\alpha} } ,
\]
where $c_1=c(r)$. Therefore,
\begin{align*}
\NN^{2\alpha-4} \norm{ R_\NN^{(r+2)} }      & \le \NN^{2\alpha-4} \norm{ \rho_\NN^{\alpha-2}}  \norm{ R_\NN^{(r+2)}\rho_\NN^{2-\alpha} }  =\norm{ R_\NN^{(r+2)}\rho_\NN^{2-\alpha} }   \\
& =\norm{ (R_\NN^{(r+2)}-R_{r+2}^{(r+2)})\rho_\NN^{2-\alpha} }
 \le c_1 \norm{ \frac{R_\NN-f}{\rho_\NN^{r+\alpha}}  }  +c_1 \norm{ \frac{f-R_{r+2}}{\rho_\NN^{r+\alpha}} }   \\
 &   \le 2c_1\norm{ \frac{f-R_{r+2}}{\rho_\NN^{r+\alpha}} }
  \le 2c_1 E_{r+2}(f) \norm{ \frac{1}{\rho_\NN^{r+\alpha}} }
 = 2c_1\NN^{2r+2\alpha}E_{r+2}(f).
\end{align*}
Hence,
\begin{align*}
\omega_2(f^{(r)},t)&\le \omega_2(g^{(r)},t) + \omega_2(R_\NN^{(r)},t)\le c(r,\alpha)t^{\alpha}+ t^2\norm{ R_\NN^{(r+2)} }\\
&\le c(r,\alpha)t^{\alpha}+2c_1\NN^{4+2r}t^2 E_{r+2}(f),
\end{align*}
and the proof is complete.
\end{proof}

 \begin{remark}
\lem{modulus} is, actually, a particular case of the following statement which is a version of the general classical Dzyadyk-Timan-Lebed'-Brudnyi inverse theorem for   approximation by algebraic polynomials (see \eg  \cite{ds}*{p. 381} for references). We preferred to give a short proof of \lem{modulus} in order that the article be self-contained. This general \thm{inversethm}  is not applied in the current paper, but it will likely be needed (especially for $k=1$ and $k=2$) for answering the open problems in \sectio{convecsection}.
\end{remark}


\begin{theorem} \label{inversethm}
Suppose that $r\in\N_0$, $k\in\N$, $\NN\in\N$,
  and   
  \[
  \phi\in\Phi^k := \left\{ \phi:[0,\infty)\rightarrow[0,\infty) \st \phi(0)=0, \; \phi \uparrow \;\; \mbox{\rm and}\;\;
  t^{-k}\phi(t) \downarrow
  \right\}
  \]
  is such that
\[
 \int_0^1 \frac{r \phi (u)}{u}du<+\infty
\]
$($\ie, if $r=0$ then this condition is vacuous$)$.
If, for $f: [-1,1]\mapsto\R$, there exists a sequence of algebraic polynomials $\{P_n\}_{n\ge \NN}$   such that
\[
\left|f(x)-P_n(x)\right| \le \rho_{n}^r(x) \phi\left( \rho_{n}(x)\right),\quad \mbox{for all} \quad  x\in [-1,1] \andd  n\ge \NN,
\]
then
$f\in C^r[-1,1]$  and, for   $0<t\leq 1/2$,
\[
\w_k( f^{(r)},t)  \le
c(k,r) \int_0^t\frac{r\phi(u)}{u }du
+ c(k,r) t^{k}\int_t^1\frac{\phi(u)}{u^{k +1}}du
 +c(k, r, \NN)t^{k}E_{k+r}(f).
\]
In particular, if $\NN\le k+r$, then
\[
\w_k( f^{(r)},t)  \le
c(k,r) \int_0^t\frac{r\phi(u)}{u }du
+ c(k, r) t^{k}\int_t^1\frac{\phi(u)}{u^{k +1}}du.
\]
\end{theorem}

  \lem{modulus} immediately follows from \thm{inversethm}
 by setting $k=2$ and $\phi(u)=u^\a$, $0<\a<2$.
The proof of \thm{inversethm}  is rather standard and will be omitted. 

\sect{Comonotone approximation: uniform estimates} \label{comsec} 

The following theorem was proved in \cite{LRS} (see also \cite{KLPS}*{Theorem 12}).

\begin{theorem}\label{thmlrs}
If $Y_s$, $s\ge1$, and $\alpha>0$ are given, and if a function
$f\in\Delta^{(1)}(Y_s)$ satisfies
\be \label{33}
n^\a E_n(f)\ge 1 ,  \quad n\ge 1 ,
\ee
  then
\[
n^\alpha E_n^{(1)}(f,Y_s)\le c(\alpha,s),\quad n\ge \NN^*,
\]
where $\NN^*=1$ if $\a\not\in A_s$ (with $A_s$ defined in \ineq{exceptional}) and
 $\NN^*=\NN^*(\alpha,Y_s)$, if $\a\in A_s$.

Moreover, this statement cannot be improved since, if $s\ge1$ and
$\a\in A_s$, then
for every $m\ge 1$, there are a collection $Y_s$ and a function
$f\in\Delta^{(1)}(Y_s)$ satisfying $\ineq{33}$ and
\[
m^\alpha E_m^{(1)}(f,Y_s)\ge c(s)\ln m .
\]
\end{theorem}

Suppose now  that we do not have  \ineq{33} and only   have knowledge about the degree of unconstrained approximation $E_n(f)$  beginning with some natural number $\NN$.
What can we say about the degree of comonotone approximation in that case?
More precisely, the question  is:
\begin{quote}
does there exist a natural number $\NN^*$   such that
\be\label{monotone}
\sup \{ n^\a E_n^{(1)}(f,Y_s) : n\ge \NN^*\} \le
c(\a,s) \sup\{ n^\a E_n(f) : n\ge \NN\},
\ee
and what parameters among $\a$, $\NN$, $Y_s$ and $f$  does it have to depend on?
\end{quote}

The answer to this question is, in general, different for each given triple $(\a,\NN,s)\in \R_+\times \N \times \N_0$ (we include the case $s=0$ for comparison).
It turns out that there are three different types of behavior of $\NN^*$ and, in order to describe them, we use the following notation:
 \begin{enumerate}

\item  We write $(\a,\NN,s)\in``+"$, if \ineq{monotone} holds with
$\NN^*=\NN$.

\item  We write $(\a,\NN,s)\in``\oplus "$, if
\begin{itemize}
\item    \ineq{monotone} holds with $\NN^*=\NN^*(\a,\NN,Y_s)$, and

\item  \ineq{monotone} is not valid with $\NN^*$ which is
independent of $Y_s$, that is, for each $A>0$ and $M\in\mathbb N$, there exist a number $m>M$, a
collection $Y_s\in\Y_s$ and a non-polynomial  $f\in\Delta^{(1)}(Y_s)$
such that
\be\label{negative}
m^\a E_m^{(1)}(f,Y_s)\ge A\sup\{ n^\a E_n(f) : n\ge \NN\} .
\ee
\end{itemize}

\item  We write $(\a,\NN,s)\in ``\ominus"$, if
\begin{itemize}
\item \ineq{monotone} holds with $\NN^*=\NN^*(\a,\NN,Y_s,f)$, and

\item  \ineq{monotone} is not valid with $\NN^*$ which is
independent of $f$, that is, for each $A>0$, $M>0$  and   $Y_s\in\Y_s$, there exist a number
$m>M$ and a  non-polynomial $f\in\Delta^{(1)}(Y_s)$, such that \ineq{negative} holds.
\end{itemize}
\end{enumerate}

\begin{remark}
We emphasize that, in the case ``$\ominus$", \ineq{monotone} is not valid with $\NN^*$ which is
independent of $f$,  {\bf for each $Y_s\in\Y_s$}. So far, we have not encountered any cases when \ineq{monotone} is not valid with $\NN^*$ which is
independent of $f$,  for some but not for all  $Y_s\in\Y_s$.

Also, while it is theoretically possible for \ineq{monotone} to hold with $\NN^*$ that depends on $\NN$ but is strictly larger than $\NN$, we have not encountered any  such cases either.
\end{remark}

The following theorem summarizes the results in \cite{LRS}, \cite{LSV} and \cite{LS1}.

\begin{theorem} \label{comonuniform}
For every pair $(\alpha,  s)\in \R_+  \times \N_0$,  there exists a constant $c(\alpha,s)$
that has the following property: If $Y_s\in\Y_s$, $f\in\Delta^{(1)}(Y_s)$,  $\NN\in\N$ and
$n^\alpha E_n(f)\le 1$, $ n\ge\NN$, then
\[
n^\alpha E_n^{(1)}(f,Y_s)\le c(\alpha,s),\quad n\ge\NN^*,
\]
 where
\begin{enumerate}[\rm 1.]
\item  $(\a,\NN,s)\in``+"$  if
\begin{itemize} 
\item
 $\alpha\notin A_s$ and $\NN\le\lceil\alpha/2\rceil$, or
\item
  $2s<\alpha\le2s+2$ and $\NN\le s+2$, or
\item
  $\alpha>2s+2$ and   $\NN\ge1$.
\end{itemize} 

\item  $(\a,\NN,s)\in ``\ominus"$  if
\begin{itemize} 
\item
  $\lceil\alpha\rceil=1$, and $s\ge1$ and $\NN\ge s+2$, or $s=0$ and $\NN\ge3$,  or

\item
  $\lceil\alpha\rceil=2$ and $\NN\ge s+3$.
\end{itemize} 

\item  $(\a,\NN,s)\in ``\oplus"$  in all other cases.
\end{enumerate}

\end{theorem}

It is easier to visualize the conclusions of \thm{comonuniform} and recognize the pattern of the behavior of the dependence of $\NN^*$ on the parameters in the tables below.

In order to illustrate the behavior of $\NN^*$ when the function changes monotonicity at least once, we need a new symbol:
\[
\opp := \begin{cases}
\oplus, \text{ if } \ \alpha\in A_s,\\
+, \text{ otherwise.}\end{cases}
\]

Note that Table~1 for monotone approximation already appeared in \cite{KLPS}*{Table 14}, and we repeat it here for easy comparison.

\begin{figure}[H]
\begin{subfigure}[b]{.5\textwidth}
\[
\begin{matrix}
 \lceil{\alpha/2}\rceil &\vdots  &\vdots&\vdots&\vdots  &\updots
      \cr 4 &+&+  &+ &+ &  \cdots&
      \cr 3 &+ &+ &+ & + &  \cdots&
      \cr 2 &+&+ &+ & +   &\cdots&
      \cr 1 &+& +& \ominus &\ominus    & \cdots&
       \cr &1 &2&3&4&\mathcal N&
\end{matrix}
\]
\caption{Table 1: $s=0$}
\end{subfigure}
\begin{subfigure}[b]{.5\textwidth}
\[
\begin{matrix}
 \lceil\alpha\rceil & \vdots & \vdots & \vdots & \vdots & \updots
      \cr 5 & + &+ & + & + & \cdots
      \cr 4 & + & + & + & \oplus & \cdots
      \cr 3 & + & + & + & \oplus & \cdots
      \cr 2 & \opp & \oplus & \oplus & \ominus & \cdots
      \cr 1 & + & \oplus & \ominus & \ominus & \cdots
      \cr & 1 & 2 & 3 & 4 & \mathcal N
\end{matrix}
\]
\caption{Table 2: $s=1$}
\end{subfigure}
\end{figure}

\begin{figure}[H]
\begin{subfigure}[b]{.5\textwidth}
\[
\begin{matrix}
 \lceil\alpha\rceil & \vdots & \vdots & \vdots & \vdots & \vdots & \updots
      \cr 7 & + &+ & + & + & + & \cdots
      \cr 6 & + &+ & + & + & \oplus & \cdots
      \cr 5 & + &+ & + & + & \oplus & \cdots
      \cr 4 & \opp & \opp & \oplus & \oplus & \oplus & \cdots
      \cr 3 & + & + & \oplus & \oplus & \oplus & \cdots
      \cr 2 & \opp & \oplus & \oplus & \oplus & \ominus
      &\cdots
      \cr 1 & \opp & \oplus & \oplus & \ominus & \ominus
      &\cdots
      \cr &1 & 2 & 3 & 4 & 5 & \mathcal N
\end{matrix}
\]
\caption{Table 3: $s=2$}
\end{subfigure}
\begin{subfigure}[b]{.5\textwidth}
\[
\begin{matrix}
 \lceil\alpha\rceil & \vdots & \vdots & \vdots & \vdots & \vdots & \vdots & \updots
      \cr 9 & + &+ & + & + & + & + & \cdots
      \cr 8 & + &+ & + & + & + & \oplus & \cdots
      \cr 7 & + &+ & + & + & + & \oplus & \cdots
      \cr 6 & \opp & \opp &
      \opp & \oplus & \oplus & \oplus & \cdots
      \cr 5 & + &+ & + & \oplus & \oplus & \oplus & \cdots
      \cr 4 & \opp & \opp & \oplus &
      \oplus & \oplus & \oplus & \cdots
      \cr 3 & + & + & \oplus & \oplus & \oplus & \oplus & \cdots
      \cr 2 & \opp & \oplus & \oplus & \oplus & \oplus & \ominus
      &\cdots
      \cr 1 & \opp & \oplus & \oplus & \oplus & \ominus & \ominus
      &\cdots
      \cr & 1 & 2 & 3 & 4 & 5& 6& \mathcal N
\end{matrix}
\]
\caption{Table 4: $s=3$}
\end{subfigure}
\end{figure}

\medskip

For the general table we require one more symbol:

\[
\opm:=\begin{cases}
\ominus , \text{ if }\ 0<\alpha\le1,\\
\oplus, \text{ if }\ 1<\alpha\le2.
\end{cases}
\]

\begin{figure}[H]
\[
\begin{array}{ccccccccccc}
\lceil\alpha/2\rceil &\vdots & \vdots &  \vdots &  \vdots & \vdots &
\vdots & \vdots & \vdots  & \vdots & \updots\\
s+2 & + & + & + & \cdots   & + & + & + & + & + &\cdots\\
s+1 & + & + & + & \cdots   & + & + & + & + & \oplus &\cdots\\
s & \opp & \opp & \opp & \cdots &
\opp & \opp & \oplus & \oplus & \oplus & \cdots\\
s-1 & \opp & \opp & \opp &
\cdots & \opp & \oplus & \oplus & \oplus & \oplus & \cdots\\
\vdots & \vdots & \vdots & \vdots & \updots & \vdots & \vdots & \vdots & \vdots & \vdots & \vdots\\
3 & \opp & \opp & \oplus & \cdots & \oplus & \oplus & \oplus & \oplus & \oplus & \cdots\\
2 & \opp & \opp & \oplus & \cdots & \oplus & \oplus & \oplus & \oplus & \oplus & \cdots\\
1 & \opp & \oplus & \oplus & \cdots   & \oplus &
\oplus &\oplus & \opm & \ominus & \cdots\\
& 1 & 2 & 3 & \cdots   & s-1 & s & s+1 & s+2 & s+3&\mathcal N\\
\end{array}
\]
\caption{Table 5: $s\ge4$}
\end{figure}

\sect{Convex and co-convex approximation: open problems and remarks} \label{convecsection}

For convex functions an analog of  \thm{pointwise} is unknown. Thus, we formulate our first open problem.

\begin{openproblem} \label{pointwise1} Prove or disprove the following statement:
\begin{quote}
 for every $r \ge 1$, there is a constant $c(r)$ such that, for each function $f\in\Delta^{(2)}\cap C^{(r)}[-1,1]$, there are a number $\NN=\NN(f,r)$
and a sequence $\{P_n\}_{n\ge \NN}$ of polynomials $P_n\in\Pn\cap\Delta^{(2)}$, satisfying \ineq{101}.
\end{quote}
\end{openproblem}

We remark that the statement in the above open problem is valid if $r=0$ and $\NN=2$ (see \eg \cite{L}). At the same time, its validity is unknown even
  if $\w_2$ in \ineq{101} is replaced with $\w_1$.

Analogs of \thm{pointnew} and \cor{point} are also open  in the convex case.

\begin{openproblem}\label{openconvex} Is it true that, for each function $f\in\Delta^{(2)}$ and every $\alpha \ge 2$, we have
\be \label{corinconvex}
\limsup_{n\to\infty}n^\alpha E_{n,\alpha}^{(2)}(f)\le c(\a) \limsup_{n\to\infty}n^\alpha E_{n,\alpha}(f) ?
\ee
\end{openproblem}
Note that it follows from \cite{L}  that, for $0<\a<2$, \ineq{corinconvex} is valid.

\medskip

We now turn our attention to uniform estimates.
For (co)convex approximation, results similar to the ones discussed in the (co)monotone case in \sectio{comsec} were previously summarized in  \cite{KLPS}*{Tables 29-31}).
Everything was resolved with
 one exception which is the entry ``$?^*$" in \cite{KLPS}*{Table 31}.
Namely, if $2<\a\le 4$ and $\NN=s+3\ge 6$, we did not know if the constant $\NN^*$ in the inequality \ineq{monotone} (with $E^{(1)}$ replaced by $E^{(2)}$) had to depend only on $Y_s$ or on $f$ as well  (and we knew that ``$?^*$" could not be replaced by anything other than ``$\oplus$" or ``$\ominus$").

We now know (see \cite{LS1}) that, for $2<\alpha<4$, we have $\NN^*=\NN^*(\alpha,Y_s)$ in this case, \ie, if $\a\ne 4$, then ``$?^*$" in  \cite{KLPS}*{Table 31} can be replaced by ``$\oplus$", and this question is still open if $\a=4$. Hence, we have the following open problem.

\begin{openproblem} \label{op4.2}
 Prove or disprove the following statement:
 \begin{quote}
If $s\ge 3$ and $Y_s\in\Y_s$, then there is a constant $c(s)$ such that, for every $f\in\Delta^{(2)}(Y_s)$, there exists
  a natural number $\NN^*=\NN^*(s, Y_s)$   such that
\be\label{conv}
\sup \{ n^4 E_n^{(2)}(f,Y_s) : n\ge \NN^*\} \le
c(s) \sup\{ n^4 E_n(f) : n\ge s+3\} .
\ee
\end{quote}
\end{openproblem}

\begin{remark} If one replaces $n\ge s+3$ by $n\ge s+2$ in \ineq{conv}, then the statement in this open problem is true. At the same time if $n\ge s+3$ is replaced by $n\ge s+4$, then this statement is, in general, not true (see \cite{KLPS}*{Table 31}). Moreover, it follows from the fact that there is ``$\ominus$'' in \cite{KLPS}*{Table 31} if $\a=4$ and $\NN=s+4$ that
 this statement is true if $\NN^*$ is allowed to depend on $f$.
\end{remark}

In the case $s=3$, even the following problem is still open, and we feel that its resolution will successfully resolve Open Problem~\ref{op4.2} which is more general.

\begin{openproblem}\label{117}
Prove or disprove that,
if a function $f\in C[-1,1]\cap C^{(2)}(-1,1)$ is such that
$x(x^2-1/4)f''(x)\ge0$, $x\in(-1,1)$,
and
$n^4 E_n(f)\le 1$, for all $n\ge 6$,
then there exists an absolute constant $\NN^*\in\N$ such that,  for each $n\ge\NN^*$, there is a polynomial $P_n\in\Pn$ satisfying
\be \label{in1}
x(x^2-1/4)P_n''(x)\ge0 , \quad x\in(-1,1) ,
\ee
and
\be \label{in2}
n^4 \norm{f-P_n} \le c ,
\ee
where $c$ is an absolute constant.
\end{openproblem}

Note that it is possible to construct a polynomial $P_6\in\Poly_6$ satisfying both \ineq{in1} and \ineq{in2} with $n=6$. However, this does not resolve this open problem.

\sect{Higher order shape constraints} \label{lastsection}

Shape preserving approximation with $q\ge 3$ is completely different from comonotone and coconvex cases.
%
%
Recall that $W^r$, $r\ge1$, denotes the Sobolev space of $(r-1)$-times differentiable  functions $f$, such that  $f^{(r-1)}$  is absolutely continuous in $[-1,1]$ and $\|f^{(r)}\|<\infty$.
It is well known that if $f\in W^r$, then
\[
E_n(f)\le c(r) \frac{\|f^{(r)}\|}{n^r},\quad n\ge r.
\]

However, in \cite{LS2}*{Theorem 1.1}, it was proved that

\begin{theorem}\label{q} For each $q\ge 3$, $r\ge1$, $s\ge1$ and any collection $Y_s\in\Y_s$, there exists a function $f\in\Delta^{(q)}(Y_s)\cap W^r$, such that
\[
\limsup_{n\to\infty}n^rE_n^{(q)}(f,Y_s)=\infty.
\]
\end{theorem}

In fact, for $q=3$ and $r\ge3$, a somewhat stronger result was proved in \cite{LS2}*{Theorem 1.6} and then  generalized in \cite{Bez}.

\begin{theorem} Let $q\ge3$, $r\ge q$, $s\ge1$ and $Y_s\in\Y_s$. Then, there is a function
$f\in\Delta^{(q)}(Y_s)\cap W^r$, such that
\[
\limsup_{n\to\infty}n^{r-q+2}E_n^{(q)}(f,Y_s)>0.
\]
\end{theorem}

We also note that more restrictive but more precise estimates were proved in \cite{LS2}*{Theorems 1.2 and 1.5}.

\begin{theorem}
Let  $q\ge3$, $s\ge1$ and $Y_s\in\Y_s$. There is a function $f\in\Delta^{(q)}(Y_s)\cap W^{q-2}$, such that
\[
E_n^{(q)}(f,Y_s)\ge c(q,Y_s), \quad n\ge1,
\]
and there is a function $f\in\Delta^{(q)}(Y_s)\cap W^{q-1}$, such that
\[
nE_n^{(q)}(f,Y_s)\ge c(q,Y_s), \quad n\ge1,
\]
where $c(q,Y_s)$ are positive constants that depend  only on $q$ and $Y_s$.
\end{theorem}

\end{document}